\documentclass[11pt]{article}
\usepackage[english]{babel}
\usepackage[utf8]{inputenc}
\usepackage[T1]{fontenc}
\usepackage{subcaption}

\usepackage[letterpaper, portrait, margin=1in,top=1in,bottom=1.5in]{geometry}

\usepackage{amsmath}
\usepackage{amssymb}
\usepackage{amsthm}
\usepackage{amsfonts}
\usepackage{mathrsfs}
\usepackage{tikz}
\usepackage{graphicx}
\usepackage[shortlabels]{enumitem} 
\usepackage{titlesec}
\usepackage{lipsum}
\usepackage{thmtools}

\usepackage{csquotes}
  
\newtheorem{theorem}{Theorem}[section]
\newtheorem{lemma}[theorem]{Lemma}
\newtheorem{corollary}{Corollary}[theorem]
\newtheorem{proposition}[theorem]{Proposition}
\theoremstyle{remark}
\newtheorem{remark}{Remark}

\title{Spectral properties of the non-backtracking matrix of a graph}
\author{Cory Glover\footnote{Department of Mathematics, Brigham Young University, Provo UT, USA, cory.s.glover@gmail.com} ~and Mark Kempton\footnote{Department of Mathematics, Brigham Young University, Provo UT, USA, mkempton@mathematics.byu.edu}}
\date{}

\begin{document}
\maketitle

\begin{abstract}
    We investigate the spectrum of the non-backtracking matrix of a graph.  In particular, we show how to obtain eigenvectors of the non-backtracking matrix in terms of eigenvectors of a smaller matrix.  Furthermore, we find an expression for the eigenvalues of the non-backtracking matrix in terms of eigenvalues of the adjacency matrix, and use this to upper-bound the spectral radius of the non-backtracking matrix, and to give a lower bound on the spectrum.  We also investigate properties of a graph that can be determined by the spectrum.  Specifically, we prove that the number of components, the number of degree 1 vertices, and whether or not the graph is bipartite are all determined by the spectrum of the non-backtracking matrix.
\end{abstract}

\section{Introduction}
Spectral techniques are ubiquitous in the study of random walks on graphs.  For instance, it is well known that the adjacency matrix of a graph can be used to enumerate walks, and eigenvalues of the transition probability matrix can be used to bound the mixing rate and mixing time of a simple random walk.  In recent years, the study of \emph{non-backtracking} random walks has gained considerable interest.  A non-backtracking random walk is a random walk on a graph with the added requirement that each step cannot travel to the vertex visited on the immediate previous step.  Just as the adjacency matrix can be used to enumerate simple walks, a matrix called the \emph{non-backtracking matrix} of a graph can be defined, which can be used to enumerate non-backtracking walks in the graph (see Section \ref{sec:nbmatrix} for details).  The non-backtracking matrix and its spectral properties have been the object of considerable study recently.
Bordenave et.~al.~\cite{7354460} and. Newman \cite{newman2013spectral} used them for community detection on a graph.
Centrality measures using non-backtracking random walks have been discussed by Arrigo et al and Lin and Zhang \cite{arrigo2018non,lin2019non}.
Pan, Jiang and Xu have used them to maximize influence on social networks \cite{7866135}.

In this paper, we study the spectral properties of the non-backtracking matrix of a graph.
It is well-known that the spectrum of a non-backtracking matrix is connected to the spectrum of the adjacency matrix using Ihara's Theorem \cite{alon2007non,kempton2016,lubetzky2016cutoff,torres2020non} (see \ref{thm:ihara}).
Specifically, many have noted the relationship between the non-backtracking spectrum and a (usually) smaller matrix we will denote $K$ (see Section 4) \cite{krzakala2013spectral,lin2019non,torres2020non}.
Bordenave et al also have studied non-backtracking spectrum of Erdos-Renyi random graphs \cite{7354460}.

The spectrum of the non-backtracking matrix is well understood for regular graphs, with explicit expressions for eigenvalues from Ihara's Theorem, and work of Lubetsky and Peres \cite{lubetzky2016cutoff} gives explicit constructions of eigenvectors for regular graphs.
Our main goals are to: 1) develop deeper understanding into the eigenvalues and eigenvectors of the non-backtracking matrix, especially for irregular graphs, 2) identify bounds on the non-backtracking spectrum of a graph, and 3) relate spectral properties of the non-backtracking matrix to structural properties of the graph.
Our first main result is to give a new proof of Ihara's Theorem. Our approach not only tells us the spectrum of the non-backtracking matrix, but also gives a decomposition of the non-backtracking matrix from which eigenvectors can be determined in terms of eigenvectors of a smaller matrix.
We furthermore use eigenvectors of the matrix $K$ mentioned above to
develop a formula that (under the right conditions) gives the eigenvalues of the non-backtracking matrix in terms of eigenvalues of the adjacency matrix for irregular graphs.
Using this expression, we 
bound the spectral radius of the non-backtracking matrix.  Finally, we prove that various structural properties of the graph, such as number of components, number of degree 1 vertices, and bipartiteness are completely determined by the non-backtracking spectrum. 

The remainder of the paper is organized as follows. In Section 2, we will define the non-backtracking matrix and review key facts known about its spectrum.  In Section 3, we will investigate the non-backtracking spectrum of trees and cycles, as well as graphs with pendant cycles.
In Section 4, we will relate the non-backtracking matrix to a block diagonal matrix.
This block diagonal construction will give our alternate proof to Ihara's Theorem.
Lastly, Section 5 will discuss upper bounds on the non-backtracking spectrum in terms of the spectral radius of the adjacency matrix.
Additionally, we will derive a lower bound on the minimum modulus of the non-backtracking spectrum.
Finally, we will identify properties of the non-backtracking spectrum of bipartite graphs.

\section{The Non-Backtracking Matrix}\label{sec:nbmatrix}
Let $G=(V,E)$ be a graph with vertices $V$ and edges $E$.
Let $n$ and $m$ be the number of vertexs and edges of $G$ respectively.
Given a starting vertex on $G$, a sample random walk of length $k$ is a collection of vertices $(v_1,...,v_k)$, where $v_i$ is chosen uniformly at random from the neighbors of $v_{i-1}$.
A non-backtracking random walk (NBRW) is a walk where $v_i$ is chosen uniformly from the set of neighbors to $v_{i-1}$ excluding $v_{i-2}$, for all $i>2$.
In order to consider a NBRW as a Markov chain, we consider an equivalent random walk along the edges of the directed graph $\hat{G}$.
We define $\hat{G}$ to have a directed edge from $i\rightarrow j$ and $j\rightarrow i$ if $i\sim j$ in $G$, and that $i\rightarrow j$ only connects to an edge $k\rightarrow l$ if $l\neq i$ and $j=k$ (see Figure \ref{fig:cycle}).
To encapsulate this Markov chain in a matrix, we define the non-backtracking matrix $B$ of $G$ such that
\begin{equation}
    B((u,v),(x,y))=\begin{cases}1&v=x\text{ and }u\neq y\\0&\text{otherwise}\end{cases}
    \label{eqn:nb-matrix-def}
\end{equation}
where $(u,v)$ and $(x,y)$ are edges between $u\sim v$ and $x\sim y$ respectively.
It has previously been proven that the spectrum of $B$ can be found using Ihara's Theorem (see \cite{ihara1966discrete,kotani20002}).
\begin{theorem}[Ihara's Theorem]\label{thm:ihara}
Given a graph $G$ with $n$ vertexs and $m$ edges, let $B$ be the non-backtracking matrix of $G$ as defined above. Let $A$ denote the adjacency matrix of $G$ and $D$ the degree matrix. Then
\begin{equation}
    \text{det}(I-uB)=(1-u^2)^{m-n}\text{det}(u^2(D-I)-uA+I).
    \label{eqn:char-B}
\end{equation}
\label{thm:ihara}
\end{theorem}

\begin{figure}[h]
    \centering
    \begin{tikzpicture}
\draw[ultra thick] (-1,0)--(0,1);
\draw[ultra thick] (0,1)--(1,0);
\draw[ultra thick] (1,0)--(0,-1);
\draw[ultra thick] (0,-1)--(-1,0);
\filldraw[black] (1,0) circle (3pt);
\filldraw[black] (-1,0) circle (3pt);
\filldraw[black] (0,1) circle (3pt);
\filldraw[black] (0,-1) circle (3pt);
\draw[ultra thick] (2,0)--(3,1);
\draw[ultra thick] (3,1)--(4,0);
\draw[ultra thick] (4,0)--(3,-1);
\draw[ultra thick] (3,-1)--(2,0);
\draw[->,ultra thick,red] (4,0) arc (0:70:1);
\draw[->,ultra thick,red] (3,1) arc (90:160:1);
\draw[->,ultra thick,red] (2,0) arc (180:250:1);
\draw[->,ultra thick,red] (3,-1) arc (270:340:1);
\draw[->,ultra thick,red] (4,0) arc (90:160:1);
\draw[->,ultra thick,red] (3,1) arc (180:250:1);
\draw[->,ultra thick,red] (2,0) arc (270:340:1);
\draw[->,ultra thick,red] (3,-1) arc (0:70:1);
\filldraw[black] (3,1) circle (3pt);
\filldraw[black] (2,0) circle (3pt);
\filldraw[black] (4,0) circle (3pt);
\filldraw[black] (3,-1) circle (3pt);
\end{tikzpicture}
\caption{A NBRW along the 4-cycle can be considered a Markov chain by performing a simple random walk along the directed edges on the right graph with a non-backtracking condition.}
\label{fig:cycle}
\end{figure}
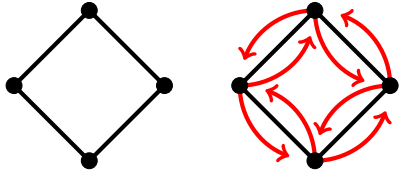

Thus the eigenvalues of $B$ all either $\pm 1$ or solutions to $\text{det}(\mu^2 I-\mu A+(D-I))=0$. Our goal is identify properties of the spectrum of $B$, denoted $\sigma(B)$, for a given graph $G$.
When $G$ is $d$-regular (each vertex has degree $d$), then Ihara's Theorem works out in a straightforward way to give the spectrum of $B$ (see \cite{alon2007non,kempton2016,lubetzky2010cutoff}).  This we state in the following theorem. Similar results have been obtained for bipartite biregular graphs \cite{kempton2016}. 
\begin{theorem}
Let $G$ be a $d$-regular graph and $A$ the adjacency matrix of $G$.
Then the eigenvalues of $B$ are
\[\pm1,\frac{\lambda_i\pm\sqrt{\lambda_i^2-4(d-1)}}{2},(i=1,...,n)\]
where $\lambda_i\in\sigma(A)$ and $\pm 1$ each have multiplicity $m-n$.
\label{kemp-reg-spectrum}
\end{theorem}

In general it is useful to know when $B$ is irreducible.
Throughout the paper, we will use this condition to employ the Perron-Frobenius theorem.

\begin{proposition}
Let $G$ be a connected graph that is not a cycle and $d_{\min}\geq 2$.
Then $B$ is irreducible.
\label{thm:irreducible}
\end{proposition}

\begin{proof}
Since $B$ represents a directed graph $\hat{G}$, it suffices to show that the graph represented by two directed edges between each vertex with the non-backtracking constraint is strongly connected.
Denote $i$ as the directed edge from $a\rightarrow b$ and $-i$ as the directed edge from $b\rightarrow a$ where $a$ and $b$ are vertices of the original graph $G$.
Note that if $G$ is connected, a simple random walk is irreducible on $G$.
Assume that a path exists on $\hat{G}$ between $i$ and $-i$ for every $i$ in $\hat{G}$.
Then a simple random walk across $\hat{G}$ will have the ability to backtrack after a finite number of steps.
Combining this with the connectedness of $G$ implies that $\hat{G}$ must be strongly connected.
Hence it is sufficient to show that there is a path from $i$ to $-i$ for every $i$ in $\hat{G}$.

Since $G$ is not a cycle, there must exist at least one vertex with degree greater than or equal to 3.
We examine a walk across $\hat{G}$ beginning at directed edge $i$.
Assume that $i$ is pointing towards a vertex $a$ of degree at least 3.
Assume that we take the shortest path from $i$ to some directed edge pointing towards $a$ which is not the trivial path of length 0.
If this path arrives at an edge $k$ pointing towards $a$ such that $k\neq i$, then we can take a step from $k$ to $-i$ since $G$ is connected.
If this path arrives at $i$, then take a step onto a directed edge $j$ which is not in the current path.
We are guaranteed such $j$ exists since $a$ has degree 2.
Take the shortest path from $j$ to another directed edge pointing towards $a$ that is not $j$ itself.
If this path arrives at some directed edge $l\neq i$, then we can take a step from $l$ to $-i$ since $G$ is connected.
If this path arrives at $i$, then at some point along the path from $j$ to $i$ we intersected at some edge $r$ from the first shortest path found.
When arriving at $r$, rather than continue to $i$ via $r$, step to $-r$.
Then by the existence of the path from $i$ to $r$, there must exist a path from $-r$ to $-i$.
Hence a path always exist between $i$ and $-i$.

Now assume that $i$ is pointing towards a vertex $b$ of degree 2. 
Then a random walk on $\hat{G}$ beginning at $i$ can only travel to one edge. Since $G$ is connected and $G$ is not a cycle, there exists a path from $i$ to some edge $k$ which points to a vertex of degree at least 3.
Thus, by the previous statement, a path exists between $k$ and $-k$. Then there is clearly a path from $-k$ to $-i$.
Hence there is a path from $i$ to $-i$.
Hence, $G$ is strongly connected and $B$ is irreducible.
\end{proof}


\section{Examples with few edges}

It is clear that as $G$ gets more and more dense, $B$ becomes larger.
However, when $m$ is close to $n$, $B$ is relatively small.
In this section, we examine the spectrum of $B$ specifically when $m\leq n$.

\subsection{Trees}

A tree is a connected graph which contains no cycles.
Trees have been found to be difficult to distinguish using the spectrum of $A$, since Schwenk showed that many large trees have the same spectrum \cite{schwenk1973almost}.
In the case of the non-backtracking matrix, all trees have the same spectrum.
To show this, we first find the characteristic polynomial of the edge adjacency matrix $C$ where the $G$ is a directed tree, with all edges pointing towards a root vertex and use this expression to find the non-backtracking spectrum of a tree.
We define $C$ as
\begin{equation}
   C((u,v),(x,y))=\begin{cases}1&v=x\\0&\text{otherwise}\end{cases}. 
   \label{eqn:c-matrix}
\end{equation}

\begin{remark}
The non-backtracking spectrum of a tree was first explicitly found by Torres \cite{torres2020non}.
His proof method uses properties of non-backtracking random walks.
We will use cofactor expansion to show the same result.
Additionally we will give an alternate proof to one of his corollaries using cofactor expansion.
\end{remark}

\begin{lemma}
Let $G$ be a directed tree with $n$ vertexs where all edges eventually point to one root vertex.
Let $C$ be the edge adjacency matrix of $G$.
Then $\text{det}(\lambda I-C)=\lambda^{n-1}$.
\label{thm:b-tree-lemma}
\end{lemma}

\begin{proof}
Note that a row representing any edge pointing directly to the root vertex of $G$ will have all 0 entries.
Choose an edge $j$ pointing to the root and perform cofactor expansion across this row of $\lambda I-C$ corresponding to $j$.
This gives $\text{det}(\lambda I-C)=\lambda\text{det}(\lambda I-C_{\neq\{j\}})$, where $C_{\neq\{j\}}$ is $C$ with the row and columns for $j$ deleted. By induction, continue on every edge pointing to the root. The new $\hat{C}$ will then consist of $k$ directed trees with all edges pointing to a root. Repeat by induction. Thus, $\text{det}(\lambda I-C)=\lambda^{n-1}$.
\end{proof}

\begin{theorem}[\cite{torres2020non}]
Let $B$ be the non-backtracking matrix of a tree $G$.
Then $\text{det}(\lambda I-B)=\lambda^{2n-1}$.
\label{thm:b-tree}
\end{theorem}

\begin{proof}
Let $G$ be a graph with $n$ vertexs and $k$ leaves.
Let the $i^{th}$ row of $B$ represent an edge pointing towards a leaf in $G$.
Every entry of this row will be 0.
We write $B$ such that the first $k$ rows represent the $k$ edges pointing towards the $k$ leaves of the graphs.
We cofactor expansion choosing the first row of $\text{det}(\lambda I-B)$.
Then $\text{det}(\lambda I-B)=\lambda\text{det}(\lambda I-B_{\neq 1})$ where $B_{\neq 1}$ is $B$ without the $1^{st}$ row or column.
Continuing we get that $\text{det}(\lambda I-B)=\lambda^k\text{det}(\lambda I-B_{\neq\{1,...,k\}})$.

Let the $i^{th}$ row of $B_{\neq\{1,...,k\}}$ represent the edge pointing to the parent of a leaf vertex. 
Since all rows representing edges pointing to leaf vertexs have been removed, the $i^{th}$ row will only have one nonzero entry $\lambda$. Thus, $\text{det}(\lambda I-B)=\lambda^{k+1}\text{det}(\lambda I-B_{\neq\{1,...,k,i\}})$.
Continue this process for the parents of all leaf vertexs and then for parents of parents.
This continues until $B_{\neq\{1,...,k,i_1,i_2,...,i_j\}}$ represents an adjacency matrix of a directed tree of $n-1$ edges.
By Lemma \ref{thm:b-tree-lemma}, we get that $\text{det}(\lambda I-B)=\lambda^{2(n-1)}$.
\end{proof} 

\begin{corollary}[\cite{torres2020non}]
Let $G$ be a graph with $m$ edges and $T$ be a tree with $n$ vertices.
Let $B$ be the non-backtracking matrix of $G$.
Define $\hat{G}$ as the graph constructed by joining $G$ and $T$ on one vertex.
Define $\hat{B}$ be the non-backtracking matrix of $\hat{G}$.
Then $\sigma(\hat{B})$ is $\sigma(B)$ along with eigenvalue 0 with algebraic multiplicity $2(n-1)$.
\label{thm:adding-trees}
\end{corollary}

\begin{proof}
Order the entries of $\hat{B}$ such that the last $2(n-1)$ entries represent edges in the tree $T$. Consider the matrix $\lambda I-\hat{B}$. We can use the same method in Theorem \ref{thm:b-tree} to find that $\text{det}(\lambda I-\hat{B})=\lambda^{2(n-1)}\text{det}(\lambda I-\hat{B}_{\neq \{2m+1,2m+2,...,2m+2(n-1)\}})=\lambda^{2(n-1)}\text{det}(\lambda I-B)$. The result follows.
\end{proof}

\subsection{Cycles}

In the case $\sigma(K)=\sigma(B)$, we have that $G$ is a cycle $C_n$.
The spectrum of the adjacency matrix of $C_n$ is known to be the $2\cos(2\pi j/n)$ for $j=\{0,...,n-1\}$ (see \cite{brouwer2011spectra}).
Similarly, the non-backtracking spectrum of $C_n$ can be explicitly calculated.
First, we note the following fact about the spectrum of directed cycles:

\begin{lemma}[\cite{brouwer2011spectra}]
Let $D_n$ be a directed $n$-cycle with adjacency matrix $A$.
Then $\sigma(A)$ consists of the $n^{th}$ complex roots of unity.
\label{directed-cycles}
\end{lemma}

In Figure \ref{fig:cycle}, we see that changing each of the edges of $C_n$ to be two directed edges creates two directed cycles $D_n$.
Given the non-backtracking condition, these two directed cycles can be considered disjoint.
Using this fact, \cite{torres2020non} calculates the non-backtracking spectrum of $C_n$ explicity.

\begin{theorem}[\cite{torres2020non}]
Let $C_n$ be an undirected cycle with $n$ vertices. Let $B$ be the non-backtracking matrix of $C_n$. Then the eigenvalues of $B$ are $e^{2\pi j/n}$ for $j=0,...,n-1$ and each eigenvalue has multiplicity 2.
\label{thm:cycle-spectrum}
\end{theorem}


With an understanding of the non-backtracking spectrum of cycles, we can expand to graphs with "pendant" cycles.
By this we mean a graph $G$ made from connecting a graph $\hat{G}$ at one vertex with a cycle $C_n$, as in Figure \ref{fig:exterior-cycle}.
We can guarantee that the non-backtracking spectra of graphs with pendant cycles contain eigenvalues from the non-backtracking spectrum of $C_n$.


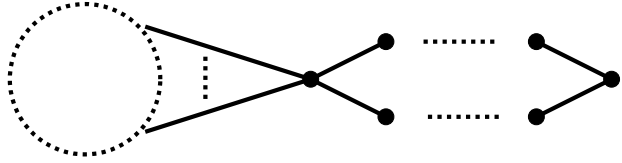
\begin{figure}
    \centering
    \begin{tikzpicture}
\draw[dotted,ultra thick] (0,0) circle (1);
\draw[ultra thick](.8,.7)--(3,0);
\draw [ultra thick] (.8,-.7)--(3,0);
\filldraw[fill=black] (3,0) circle (3pt);
\draw[dotted,ultra thick] (1.6,.3)--(1.6,-.3);
\draw[ultra thick] (3,0) -- (4,.5);
\draw[ultra thick] (3,0) -- (4,-.5);
\filldraw[fill=black] (4,.5) circle (3pt);
\filldraw[fill=black] (4,-.5) circle (3pt);
\draw[dotted,ultra thick] (4.5,.5) -- (5.5,.5);
\draw[ultra thick] (6,.5) -- (7,0);
\draw [ultra thick] (7,0) -- (6,-.5);
\draw [dotted,ultra thick] (5.5,-.5) -- (4.5,-.5);
\filldraw[fill=black] (6,.5) circle (3pt);
\filldraw[fill=black] (7,0) circle (3pt);
\filldraw[fill=black] (6,-.5) circle (3pt);
\end{tikzpicture}
    \caption{A graph with an pendant cycle takes any graph and attachs a cycle of length $n$ to one vertex of the graph.}
    \label{fig:exterior-cycle}
\end{figure}

\begin{theorem}
Let $\hat{G}$ be a graph with $2m$ edges and let $C_n$ be a cycle. Let $G$ be a graph created by joining $\hat{G}$ and $C_n$ at exactly one vertex. Let $B$ be the non-backtracking matrix of $G$.
Then $e^{2\pi j/n}\in\sigma(B)$ for $j=0,...,n-1$.
\end{theorem}

\begin{proof}
We construct $B$ such that the first $2m$ rows and columns correspond to the edges in the graph $\hat{G}$ and the last $2n$ rows and columns correspond to the edges in $C_n$.
We know that the non-backtracking matrix of a cycle can be written as a block diagonal matrix. So
\[B=\begin{bmatrix}B_1&Q&Q\\\textbf{*}&D_n&0\\\textbf{*}&0&D_n\end{bmatrix}\]
where $B_1$ is the non-backtracking matrix of $\hat{G}$ and the bottom-right $2\times 2$ block matrix is the non-backtracking matrix of $C_n$.
Since $\hat{G}$ and $C_n$ are joined at exactly one vertex, there are $d_r$ rows in the first $2m$ rows of $B$ that have nonzero entries in $Q$, where $d_r$ is the degree of the vertex joining $\hat{G}$ and $C_n$.
In fact, we can define $Q$ as a block of zeros with one column containing non-zero entries.
This column will be identical in the block above each $D_n$ block, as each edge pointing towards the cycle points to each of the two directed cycles formed on the directed edges.
We define a vector \begin{align*}
\mathbf{x}&=\begin{bmatrix}0&\dotsb&0&\mathbf{y}&-\mathbf{y}\end{bmatrix}^T\end{align*} where the first $2m$ entries are 0, and $\mathbf{y}=\begin{bmatrix}1&e^{2\pi j i/n}&e^{2(2)\pi j i/n}&\dotsb&e^{2(n-1)\pi j i/n}\end{bmatrix}^T$.
We then we recall that $D_n\mathbf{y}=e^{2\pi i j/n}\mathbf{y}$.
Thus we see that
\begin{align*}
    B\mathbf{x}&=\begin{bmatrix}B_1&Q&Q\\\textbf{*}&D_n&0\\\textbf{*}&0&D_n\end{bmatrix}\begin{bmatrix}0\\\mathbf{y}\\-\mathbf{y}\end{bmatrix}
    =\begin{bmatrix}Q\mathbf{y}-Q\mathbf{y}\\D_n\mathbf{y}\\-D_n\mathbf{y}\end{bmatrix}
    =\begin{bmatrix}0\\e^{2\pi i j/n}\mathbf{y}\\-e^{2\pi i j/n}\mathbf{y}\end{bmatrix}
    =e^{2\pi i j/n}\begin{bmatrix}0\\\mathbf{y}\\-\mathbf{y}\end{bmatrix}
    =e^{2\pi i j/n}\mathbf{x}.
\end{align*}
So $e^{2\pi i j/n}\in\sigma(B)$.
\end{proof}

We define a specific subset of these graphs with pendant cycles where both $\hat{G}$ and $C_n$ are cycles.
We will call these graphs pinwheel graphs.
An example of a pinwheel graph can be found in Figure \ref{fig:pinwheel}.

\begin{figure}
    \centering
    \begin{tikzpicture}
\draw[ultra thick](-1,-.5)--(-1,.5);
\draw [ultra thick] (1,.5)--(1,-.5);
\draw [ultra thick](1,.5)--(0,0);
\draw [ultra thick] (1,-.5) -- (0,0);
\draw [ultra thick] (-1,-.5)--(0,0);
\draw [ultra thick] (-1,.5)--(0,0);
\filldraw[fill=black] (0,0) circle (3pt);
\filldraw[fill=black] (-1,-.5) circle (3pt);
\filldraw[fill=black](-1,.5) circle (3pt);
\filldraw[fill=black](1,-.5) circle (3pt);
\filldraw[fill=black](1,.5) circle (3pt);
\end{tikzpicture}
    \caption{An example of a pinwheel graph with two 3-cycles connected at one vertex.}
    \label{fig:pinwheel}
\end{figure}
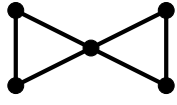

\begin{corollary}
Let $G$ be a pinwheel graph made by connecting $p$ cycles of length $k$ at one vertex. Let $B$ be its non-backtracking matrix.
Then $e^{2\pi i j/k}\in\sigma(B)$ with multiplicity $p$ for $j\in\{0,...,k-1\}$, $e^{2\pi i j/(2k)}$ with multiplicity $p-1$ for all $j\in\{1,3,5,...,2k-1\}$, and the $k^{th}$ complex roots of $2p-1$ are in $\sigma(B)$ with multiplicity 1. These capture the entirety of $\sigma(B)$ and the spectral radius of $B$ is $|(2p-1)^{1/k}|$.
\end{corollary}

\begin{proof}
We can write $B$ such that
\begin{align*}
    B=&\begin{bmatrix}C_k&0&R&R&\dotsb&R&R\\
    0&C_k&R&R&\dotsb&R&R\\
    R&R&C_k&0&\dotsb&R&R\\
    R&R&0&C_k&\dotsb&R&R\\
    \vdots&\vdots&\vdots&\vdots&\ddots&\vdots&\vdots\\
    R&R&R&R&\dotsb&C_k&0\\
    R&R&R&R&\dotsb&0&C_k\end{bmatrix}
\end{align*}
where $R$ is a block of all zeros, with the exception of the last row, which has a 1 in the first entry of the row, and zeroes everywhere else. By straightforward computation, we can see that 
    \[\begin{bmatrix}\mathbf{y}&-\mathbf{y}&0&0&0&0&\dotsb&0&0\end{bmatrix}^T\]
    \[\begin{bmatrix}0&0&\mathbf{y}&-\mathbf{y}&0&0&\dotsb&0&0\end{bmatrix}^T\]
    \[\vdots\]
    \[\begin{bmatrix}0&0&0&0&0&0&\dotsb&\mathbf{y}&-\mathbf{y}\end{bmatrix}^T\]
are all eigenvectors with eigenvalue $e^{2\pi ij/k}$ where $C_k\mathbf{y}=e^{2\pi ij/k}\mathbf{y}$. Thus, each root of unity $e^{2\pi ij/k}$ has at least algebraic multiplicity $p$.

Let $C_{2k}\mathbf{z}=e^{2\pi ij/2k}\mathbf{z}$ for $j\in\{1,3,5,...,2k-1\}$ and $\mathbf{z}=\begin{bmatrix}z_1&z_2\end{bmatrix}^T$ where $z_i$ is the $i^{th}$ half of the eigenvector $\mathbf{z}$. Then $z_1=\begin{bmatrix}\alpha_1&\dotsb&\alpha_k\end{bmatrix}^T$ and $z_2=\begin{bmatrix}\beta_1&\dotsb&\beta_k\end{bmatrix}^T$.
Since $C_{2k}$ is a circulant matrix, then $\alpha_l=e^{2\pi ij/2k}\alpha_{l+1}$ for all $l\in\{1,...,k-1\}$, $\beta_l=e^{2\pi ij/2k}\beta_{l+1}$ for all $l\in\{1,...,k-1\}$, $\alpha_k=e^{2\pi ij/2k}\beta_1$, and $\beta_k=e^{2\pi ij/2k}\alpha_1$. 
Lastly define the vector
\[\mathbf{v}_1=\begin{bmatrix}(p-1)z_1&-(p-1)z_2&-z_1&z_2&-z_1&z_2&\dotsb&-z_1&z_2\end{bmatrix}^T.\]
Then we see through direct calculation that
\begin{align*}
    B\mathbf{v}_1
    =\begin{bmatrix}-(p-1)\mathbf{r}&
    \mathbf{r}&\dotsb&\mathbf{r}\end{bmatrix}
    =e^{2\pi ij/2k}\mathbf{v}_1
\end{align*}
where $\mathbf{r}=\begin{bmatrix}\mathbf{r}_1&\mathbf{r}_2\end{bmatrix}^T$, $\mathbf{r}_1=\begin{bmatrix}-\alpha_2&\dotsb&-\alpha_k&-\beta_1\end{bmatrix}^T$ and $\mathbf{r}_2=\begin{bmatrix}\beta_2&\dotsb&\beta_k&\alpha_1\end{bmatrix}^T$.
So $\mathbf{v}_1$ is an eigenvector of $B$ with associated eigenvalue $e^{2\pi ij/2k}$. Similarly, we define \[\mathbf{v}_2=\begin{bmatrix}-z_1&z_2&(p-1)z_1&-(p-1)z_2&-z_1&z_2&\dotsb&-z_1&z_2\end{bmatrix}\] and so on for $\mathbf{v}_i$ where $i\in\{1,...,p\}$.
By a similar calculation, we see that $B\mathbf{v}_i=e^{2\pi ij/2k}\mathbf{v}_i$ for all $i$.
We now show the set $\{\mathbf{v}_1,...,\mathbf{v}_{p-1}\}$ is linearly independent.
Let $0=\sum_{i=1}^{n-1}\gamma_i\mathbf{v}_i$. Note that last two entries of this summation give that $\sum_{i=1}^{p-1}\gamma_iz_1=0$ and $\sum_{i=1}^{p-1}\gamma_iz_2=0$. Then the first entry gives that $\gamma_1(p-1)z_1=\sum_{i=2}^{p-1}\gamma_iz_1$. From the second to last entry, this must mean that $\gamma_1(p-1)z_1=-\gamma_1z_1$. Since $z_1\neq 0$ and $p-1$ is fixed, then $\gamma_1=0$. The third entry then gives that $\gamma_2(p-1)z_1=\sum_{i=3}^{p-1}\gamma_iz_1$. Again, by the second to last entry, we get that $\gamma_2(p-1)z_1=-\gamma_2z_1$. So $\gamma_2=0$. Continuing through every odd entry in order of $\sum_{i=1}^{p-1}\gamma_i\mathbf{v}_i$, we get that $\gamma_i=0$ for all $i=\{1,...,p-1\}$. Thus, the set $\{\mathbf{v}_1,\dotsb,\mathbf{v}_{p-1}\}$ must be linearly independent. Thus, $e^{2\pi ij/2k}$ must have algebraic multiplicity at least $p-1$.

We now let $\lambda=(2p-1)^{1/k}$ be one of the $k^{th}$ complex roots of $2p-1$. Then we define $\mathbf{w}=\begin{bmatrix}\lambda&\lambda^2&\dotsb \lambda^k&\lambda&\dotsb&\lambda^k&\dotsb&\lambda&\dotsb&\lambda^k\end{bmatrix}$ where $\mathbf{w}\in\mathbb{C}^{2pk}$. Then we see that
\begin{align*}
    B\mathbf{w}&=\begin{bmatrix}\lambda^2&\dotsb&\lambda^{k+1}&\lambda^2&\dotsb&\lambda^{k+1}&\dotsb&\lambda^2&\dotsb&\lambda^{k+1}\end{bmatrix}
    =\lambda\mathbf{w}.
\end{align*}
Since there are $k$ complex roots of $2p-1$, there are $k$ unique eigenvalues with at least algebraic multiplicity 1.
Further since $\mathbf{w}$ is positive and $B$ is nonnegative and irreducible by Propostion \ref{thm:irreducible}, then $|(2p-1)^{1/k}|$ must be the spectral radius of $B$.

To summarize, we have the following eigenvalues: $e^{2\pi ij/k}$ for $j\in\{0,...,k-1\}$ with multiplicity at least $p$, $e^{2\pi il/2k}$ for $l\in\{1,3,...,2k-1\}$ with multiplicity at least $p-1$, and the $k$ roots $(2p-1)^{1/k}$ with multiplicity at least $1$. We then see that $pk+(p-1)k+k=2pk$. Since $B\in M_{2pk}$, then all the multiplicities mentioned must be exact.
\end{proof}


\section{The $K$ Matrix}

The eigenvalues of $B$ coming from the equation \ref{eqn:char-B} in Ihara's Theorem can be found independent of $B$.
We define
\[K=\begin{bmatrix}A&D-I\\-I&0\end{bmatrix}\]
which has characteristic polynomial $\text{det}(\mu^2I-\mu A+(D-I))$.
Using $K$, we can create a decomposition of $B$ that clearly organizes the spectrum of $B$ and gives an alternate proof of Ihara's Theorem.

\begin{remark}
Lubetzky and Peres showed the matrix $B$ is unitarily similar to a block diagonal matrix showing which clearly shows the spectrum of $B$ as well as the eigenvectors.
\begin{theorem}[\cite{lubetzky2016cutoff}]
Let $G$ be a connected $d$-regular graph ($d\geq 3)$ on $n$ vertices.
Let $N=dn$ and $\lambda_i\in\sigma(A)$, with $\lambda_1=d$.
Then the operator $B$ is unitarily similar to 
\[\Lambda=\text{diag}\Biggl(d-1,\begin{bmatrix}\theta_2&\alpha_2\\0&\theta_2'\end{bmatrix},...,\begin{bmatrix}\theta_n&\alpha_n\\0&\theta_n'\end{bmatrix},-1,...,-1,1,...,1\Biggr)\]
where $|\alpha_i|<2(d-1)$ for $i$, $\theta_i$ and $\theta_i'$ are defined as the solutions of
\[\theta^2-\lambda_i\theta+d-1=0\]
and $-1$ has multiplicity $N/2-n$ and $1$ has multiplicity $N/2-n+1$.
\end{theorem}
While we are unable to make such a similarity for a general graph $G$, we will decompose $B$ in a similar manner into a block diagonal matrix to see the eigenvalues more clearly.
\end{remark}

\begin{remark}
Torres succeeded in diagonalizing $B$ if the eigenvalues $\mu\in\sigma(B)$ such that $|\mu|>1$ are simple.
We will not use this requirement in our proof.
\end{remark}

To create this decomposition, we need to relate the matrices $B$ and $K$ outside of just their spectrum.  
Define $S\in M_{2m\times n}$ and $T\in M_{n\times 2m}$ where
\begin{align}
    S((u,v),x)=\begin{cases}1&v=x\\0&\text{otherwise}\end{cases}&&T(x,(u,v))=\begin{cases}1&x=u\\0&\text{otherwise}\end{cases}.
    \label{eqn:s-t}
\end{align}
Define $\tau\in M_{2m\times 2m}$ to be the non-backtracking operator
\begin{align}
    \tau((u,v),(x,y))&=\begin{cases}1&v=x\text{ and }u=y\\0&\text{otherwise}\end{cases}.
    \label{eqn:tau}
\end{align}
Matrix multiplication then gives the following identities: 
\begin{align}
    C&=ST,&
    B&=ST-\tau,&
    D&=T\tau S,&
    A&=TS.
    \label{eqn:s-t-identities}
\end{align}
Using these identities it is clear that
\begin{equation}
    B\begin{bmatrix}S&T^T\end{bmatrix}=\begin{bmatrix}S&T^T\end{bmatrix}K.
    \label{eqn:b-relates-k}
\end{equation}

Before creating our decomposition, we also need to understand the eigenvectors of $B$ for eigenvalues $\pm 1$.
Lubetzky and Peres \cite{lubetzky2016cutoff} show that these eigenvectors come from $\mathscr{E}_{-1}\cap\text{Null(ST)}$ and $\mathscr{E}_1\cap \text{Null}(ST)$ respectively, where $\mathscr{E}_i$ is the eigenspace of $\tau$ corresponding to eigenvalue $i$.
They further show that $\text{dim}(\mathscr{E}_{-1}\cap\text{Null(ST)})=m-n+1$ and $\text{dim}(\mathscr{E}_{1}\cap\text{Null(ST)})=m-n$ or $m-n+1$ if $G$ is bipartite.
With this we create our decomposition.

\begin{theorem}
Let $G$ be a connected graph and $B$ its non-backtracking matrix.
Let $R\in M_{2m\times 2(m-n)}$ where the columns of $R$ are linearly independent and the first $m-n$ columns of $R$ are taken from $\mathscr{E}_{-1}\cap\text{Null}(ST)$ and the rest are taken from $\mathscr{E}_1\cap\text{Null}(ST)$.
Then
\[BX=X\left[\begin{array}{ccc}
K&0&0\\
0&I_{m-n}&0\\
0&0&-I_{m-n}
\end{array}\right]\]
and $X=\begin{bmatrix}S&T^T&R\end{bmatrix}$.
\label{thm:almost-similarity}
\end{theorem}

\begin{proof}
This follows directly from matrix multiplication and the properties of the columns of $R$.
\end{proof}

\begin{remark}
It clearly follows from the previous theorem that $\text{det}(\mu I-B)$ is just the characteristic polynomial of this block diagonal decomposition.
Hence, Ihara's Theorem is an immediate corollary of the above theorem.
Even more, Theorem \ref{thm:almost-similarity} does not just give the eigenvalues of $B$ but also the eigenvectors.
Let $\mathbf{x}$ be an eigenvector of $K$.
Then
\begin{align*}
    BX\begin{bmatrix}\mathbf{x}\\0\\0\end{bmatrix}&=X\begin{bmatrix}K&0&0\\0&I_{m-n}&0\\0&0&-I_{m-n}\end{bmatrix}\begin{bmatrix}x\\0\\0\end{bmatrix}=X\begin{bmatrix}K\mathbf{x}\\0\\0\end{bmatrix}=\mu X\begin{bmatrix}\mathbf{x}\\0\\0\end{bmatrix}.
\end{align*}
Hence we know all the eigenvectors of $B$ associated with eigenvalues of $K$ are of the form $X\begin{bmatrix}\mathbf{x}&0&0\end{bmatrix}^T$.
Additionally, from the construction of $R$, we see that the eigenvectors for $\pm 1$ are $\mathbf{y}\in\mathscr{E}_{-1}\cap\text{Null}(ST)$ and $\mathbf{z}\in\mathscr{E}_{1}\cap\text{Null}(ST)$ respectively.
\end{remark}

\section{Properties of $\sigma(B)$ using $\sigma(K)$}

With a decomposition of $B$ in terms of $K$, we now want to better understand the eigenvalues of $K$.
Our goal in this section is to use properties of $K$ in order to place bounds on the eigenvalues of $B$.
Using the Gershgorin Theorem we immediately have a bound on the spectral radius of $B$ in terms of the degree of $G$ \cite{horn2012matrix}.
\begin{proposition}
Let $G$ be a connected graph with $B$ the non-backtracking matrix and $d_{\max}$ the maximum of degree of $G$.
Then $\rho(B)\leq d_{\max}-1$ with equality if and only if $G$ is regular.
\end{proposition}

We will show a stronger bound using the matrix $K$ and its relationship with $B$.
Immediately from the structure of $K$, we can learn some information about its eigenvalue-eigenvector pairs.

\begin{proposition}
Let $G$ be a graph and $K$ as defined above. Then the following are true:
\begin{enumerate}[(i)]
\item Every eigenvector of $K$ is of the form $\begin{bmatrix}-\mu \mathbf{y}&\mathbf{y}\end{bmatrix}^T$ where $\mu\in\sigma(K)$,
\item $1\in\sigma(K)$ with geometric multiplicity equal to the number of connected components of $G$,
\item the nullity of $K$ is the number of degree 1 vertices, and
\item $K$ is invertible with inverse $K^{-1}=\begin{bmatrix}0&-I\\(D-I)^{-1}&(D-I)^{-1}A\end{bmatrix}$ if and only if $d_{\min}\geq 2$.
\end{enumerate}
\label{thm:props-of-k}
\end{proposition}

\begin{proof}
\begin{enumerate}[(i)]
\item Assume that $\mu\in\sigma(K)$ with eigenvector $\begin{bmatrix}\mathbf{x}&\mathbf{y}\end{bmatrix}^T$. Then
\begin{align*}
    \mu\begin{bmatrix}\mathbf{x}\\\mathbf{y}\end{bmatrix}=\begin{bmatrix}A&D-I\\-I&0\end{bmatrix}\begin{bmatrix}\mathbf{x}\\\mathbf{y}\end{bmatrix}=\begin{bmatrix}A\mathbf{x}+(D-I)\mathbf{y}\\-\mathbf{x}\end{bmatrix}.
\end{align*}
Then bottom block gives that $\mathbf{x}=-\mu \mathbf{y}$, so the eigenvector must be of the form $\begin{bmatrix}-\mu \mathbf{y}&\mathbf{y}\end{bmatrix}$.
\item Let $\mathbf{x}=\begin{bmatrix}\mathbf{1}&-\mathbf{1}\end{bmatrix}$ where $\mathbf{1}\in M_{n\times 1}$ is an all-ones vector. Then
\begin{align*}
    K\mathbf{x}=\begin{bmatrix}A\mathbf{1}-(D-I)\mathbf{1}\\-\mathbf{1}\end{bmatrix}=\begin{bmatrix}\mathbf{k}-\mathbf{k}+\mathbf{1}\\-\mathbf{1}\end{bmatrix}=\mathbf{x}
\end{align*}
where $\mathbf{k}$ is the degree vector of $G$. Thus, $K\mathbf{x}=\mathbf{x}$ and $1\in\sigma(K)$.

Let $(1,\begin{bmatrix}-\mathbf{y}&\mathbf{y}\end{bmatrix}^T)$ be a general eigenvalue-eigenvector pair for $1\in\sigma(K)$.
Then $\mathbf{y}-A\mathbf{y}+(D-I)\mathbf{y}=0$.
Rearranging we see that $(D-A)\mathbf{y}=0$.
So $\mathbf{y}\in\text{Null}(L)$ where $L$ is the Laplacian of $G$.
So $\text{mult}(1)\leq\text{Nullity}(L)$.

Now assume that $\mathbf{z}\in\text{Null}(L)$.
Then $(D-A)\mathbf{z}=0$.
So $\mathbf{z}-A\mathbf{z}+(D-I)\mathbf{z}=0$.
Then we see that 
\[K\begin{pmatrix}-\mathbf{z}&\mathbf{z}\end{pmatrix}^T=\begin{pmatrix}-A\mathbf{z}+(D-I)\mathbf{z}\\\mathbf{z}\end{pmatrix}=\begin{pmatrix}-\mathbf{z}&\mathbf{z}\end{pmatrix}^T.\]
Thus the geometric multiplicity of $1\in\sigma(K)$ is the nullity of $L$.
By well-known properties of the Laplacian (see \cite{brouwer2011spectra}), the geometric multiplicity of $1\in\sigma(K)$ is the number of connected components in $G$.
\item We have \[\begin{bmatrix}A&D-I\\-I&0\end{bmatrix}\begin{bmatrix}0\\\mathbf{y}\end{bmatrix}=\begin{bmatrix}0\\0\end{bmatrix}\]
if and only if \[\begin{bmatrix}(D-I)\mathbf{y}\\0\end{bmatrix}=\begin{bmatrix}0\\0\end{bmatrix}.\]  Thus $[0 \ \mathbf{y}]^T$ is in the nullspace of $K$ if and only if $(D-I)\mathbf{y}=0$.  Note that $D-I$ is diagonal, and so the dimension of its nullspace is equal to the number of diagonal entries that are 0.  From this, the result follows.
\item Immediately from (iii) we that $K$ is invertible if and only if $d_{\min}\geq 2$. If $d_{\min}\geq 2$, then
\begin{align*}
    \begin{bmatrix}A&D-I\\-I&0\end{bmatrix}\begin{bmatrix}0&-I\\(D-I)^{-1}&(D-I)^{-1}A\end{bmatrix}&=\begin{bmatrix}I&0\\0&I\end{bmatrix}\\
    \begin{bmatrix}0&-I\\(D-I)^{-1}&(D-I)^{-1}A\end{bmatrix}\begin{bmatrix}A&D-I\\-I&0\end{bmatrix}&=\begin{bmatrix}I&0\\0&I\end{bmatrix}.
\end{align*}
So $K^{-1}=\begin{bmatrix}0&-I\\(D-I)^{-1}&(D-I)^{-1}A\end{bmatrix}.$ 
\end{enumerate}
\end{proof}

\begin{remark}
Proposition \ref{thm:props-of-k}(iii) can also be proved using the relationship between $B$ and $K$ and results about the invertibility of $B$ found in \cite{torres2020non}.
\end{remark}

With Proposition $\ref{thm:props-of-k}$(i), we can create a relationship between the eigenvalues of $A$ and the eigenvalues of $B$.
\begin{proposition}
Let $\mu\in\sigma(K)$ with eigenvector $\begin{bmatrix}-\mu \mathbf{y}&\mathbf{y}\end{bmatrix}^T$ and let $\lambda\in\sigma(A)$ such that $A\mathbf{x}=\lambda \mathbf{x}$.
If $\mathbf{x}^T\mathbf{y}\neq 0$, then
\[\mu=\frac{\lambda\pm\sqrt{\lambda^2-4\mathbf{x}^T(D-I)\mathbf{y}}}{2}.\]
\label{thm:mu-equation}
\end{proposition}

\begin{proof}
Recall that $\mu^2\mathbf{y}-\mu A\mathbf{y}+(D-I)\mathbf{y}=0$ where $\mathbf{y}$ is the second half of the eigenvector $\begin{bmatrix}-\mu \mathbf{y}&\mathbf{y}\end{bmatrix}^T$ of $K$ corresponding to eigenvalue $\mu$.
Let $\mathbf{x}^T$ be an eigenvector of $A$ with associated eigenvalue $\lambda$, where $\mathbf{x}$ and $\mathbf{y}$ are not orthogonal, and scale $\mathbf{y}$ such that $\mathbf{x}^T\mathbf{y}=1$.
Then left multiplying by $\mathbf{x}^T$ gives $\mu^2-\mu\lambda+\mathbf{x}^T(D-I)\mathbf{y}=0$.
With the quadratic formula, we know that
\[\mu=\frac{\lambda\pm\sqrt{\lambda^2-4\mathbf{x}^T(D-I)\mathbf{y}}}{2}.\]
\end{proof}

\begin{remark}
Note that this formula does not necessarily give every eigenvalue $\mu$.
If $K$ is not diagonalizable (for example, $K$ of any cycle), then there will exist $\mu\in\sigma(K)$ with no corresponding eigenvector $\begin{bmatrix}-\mu \mathbf{y}&\mathbf{y}\end{bmatrix}^T$.
Additionally we are not guaranteed that $\mathbf{x}^T\mathbf{y}\neq 0$ for all pairs $(\mathbf{x},\mathbf{y})$.
In fact many $\mathbf{x}$ and $\mathbf{y}$ exist such that $\mathbf{x}^T\mathbf{y}=0$.
However the eigenvectors of $A$ form a basis of $\mathbf{R}^n$, so for each $\mu\in\sigma(K)$ with eigenvector $\begin{bmatrix}-\mu \mathbf{y}&\mathbf{y}\end{bmatrix}^T$ there exists some $i$ such that $A\mathbf{x}_i=\lambda_i \mathbf{x}_i$ where $\mathbf{x}_i^T\mathbf{y}\neq 0$.
\end{remark}

\begin{remark}
In the case of a $d$-regular graph $D-I=(d-1)I$ which allows us to use $K$ in order to provide an alternate proof to Theorem \ref{kemp-reg-spectrum}.
To do this we need to ensure $K$ is diagonalizable.
Assume that $d\geq 2$ for a given graph $G$.
Let $\mathbf{x}$ be an eigenvector of $A$ with associated eigenvalue $\lambda$. Define $\mu_1=\frac{\lambda+\sqrt{\lambda^2-4(d-1)}}{2}$ and $\mu_2=\frac{\lambda-\sqrt{\lambda^2-4(d-1)}}{2}$. This implies that $\mu_1^2-\mu_1\lambda+(d-1)=0$ and $\mu_2^2-\mu_2\lambda+(d-1)=0$. Define the vector $\mathbf{v}_i=\begin{bmatrix}-\mu_i x&x\end{bmatrix}$. Then we see that
\begin{align*}
    K\mathbf{v}_i&=\begin{bmatrix}A&(d-1)I\\-I&0\end{bmatrix}\begin{bmatrix}-\mu_i \mathbf{x}\\\mathbf{x}\end{bmatrix}
    =\begin{bmatrix}-\mu_i \lambda \mathbf{x}+(d-1)\mathbf{x}\\\mu_i \mathbf{x}\end{bmatrix}
    =\begin{bmatrix}-\mu_i^2\mathbf{x}\\\mu_i \mathbf{x}\end{bmatrix}
    =\mu_i\mathbf{v}_i.
\end{align*}
Thus, both $\mathbf{v}_1$ and $\mathbf{v}_2$ are eigenvectors. Since $\mu_1\neq\mu_2$, we know that $\mathbf{v}_1\neq\mathbf{v}_2$. Further, we know that $A$ is diagonalizable. Thus, $\mu_i$ has the same algebraic and geometric multiplicity. Lastly, there are $2n$ distinct $\mu_i$ since each eigenvector $\mathbf{x}$ of $A$ creates two unique eigenvalue-eigenvector pairs. Since $K\in M_{2n}$, $K$ is diagonalizable.

Now that we know $K$ is diagonalizable, define $f(x)=x+\frac{1}{x}(d-1)$.
Let $\mu\in\sigma(K)$ be an eigenvalue of $K$.
Recall all eigenvectors of $K$ can be written as $\begin{bmatrix}-\mu \mathbf{y}&\mathbf{y}\end{bmatrix}^T$, implying $-\mu A\mathbf{y}+(d-1)\mathbf{y}=-\mu^2\mathbf{y}$. 
Also note that $d\geq 2$, so by Proposition \ref{thm:props-of-k}(iii) $\mu\neq 0$ and $A\mathbf{y}=(\mu+\frac{1}{\mu}(d-1))\mathbf{y}$. So there exists some $\lambda\in\sigma(A)$ such that $\mu+\frac{1}{\mu}(d-1)=\lambda$. Thus, $f(\sigma(K))\subset\sigma(A)$.

We also know there are two solutions to the equation $\mu+\frac{1}{\mu}(d-1)=\lambda$:
\[\mu=\frac{\lambda\pm\sqrt{\lambda^2-4(d-1)}}{2}.\]
Let $\mu_1,\mu_2$ be the two solutions to this equation where $\mu_1$ and $\mu_2$ are the plus and minus solutions respectively.
Assume that $\begin{bmatrix}-\mu_1\mathbf{y}&\mathbf{y}\end{bmatrix}^T$ is an eigenvector of $K$. 
Then we know that $A\mathbf{y}=(\mu_1+\frac{1}{\mu_1}(d-1))\mathbf{y}$ and hence $A\mathbf{y}=(\mu_2+\frac{1}{\mu_2}(d-1))\mathbf{y}$. 
Rearranging we get $-\mu_2^2\mathbf{y}=-\mu_2A\mathbf{y}+(d-1)\mathbf{y}$. 
So, $K\begin{bmatrix}-\mu_2\mathbf{y}&\mathbf{y}\end{bmatrix}^T=\mu_2\begin{bmatrix}-\mu_2\mathbf{y}&\mathbf{y}\end{bmatrix}^T$. 
Similarly, if $\begin{bmatrix}-\mu_2\mathbf{y}&\mathbf{y}\end{bmatrix}^T$ is an eigenvector of $K$ with eigenvalue $\mu_2$, then $\begin{bmatrix}-\mu_1\mathbf{y}&\mathbf{y}\end{bmatrix}^T$ is an eigenvector of $K$.
Since $K$ is diagonalizable, then $\mu_1$ and $\mu_2$ have the same multiplicity.
Let $\text{mult}(\mu_i)$ be the algebraic multiplicity of $\mu_i$. 
Order the eigenvalues of $K$ such that if $\mu$ and $\hat{\mu}$ are both solutions of $\lambda=\mu+\frac{1}{\mu}(d-1)$, then $\mu=\mu_i$ and $\hat{\mu}=\mu_{-i}$. 
Then we know that $2n=\sum_{i=1}^{2n}\text{mult}(\mu_i)=\sum_{i=1}^n2(\text{mult}(\mu_i))=2\sum_{i=1}^n\text{mult}(\mu_i)$. 
So $n=\sum_{i=1}^n\text{mult}(\mu_i)$. 
Then we know that $|f(\sigma(K))|=\sum_{i=1}^n\text{mult}(\mu_i)=n$. 
Since $|\sigma(A)|=n$ and $f(\sigma(K))\subseteq\sigma(A)$, $\sigma(A)=f(\sigma(K))$.
\end{remark}

Our goal now is to use Proposition \ref{thm:mu-equation} to bound the spectrum of $B$.
We begin by bounding the "bottom" of the spectrum using a formula similar to that of Proposition \ref{thm:mu-equation} but not quite equal.

\begin{theorem}
Let $G$ be a connected graph where $d_{\min}\geq 2$. Then $|\mu|\geq 1$ for all $\mu\in\sigma(B)$.
\label{thm:smallest-eig}
\end{theorem}

\begin{proof}
Recall that any $\mu\neq\pm 1$ satisfies the equation $\mu^2\mathbf{y}-\mu A\mathbf{y}+(D-I)\mathbf{y}=0$ for some vector $\mathbf{y}\neq 0$.
Assume that $\|\mathbf{y}\|_2=1$.
Then left multiplying by $\mathbf{y}^T$ gives
$\mu^2-\mu\mathbf{y}^TA\mathbf{y}+\mathbf{y}^T(D-I)\mathbf{y}=0.$
The quadratic equation then gives
\[\mu=\frac{\mathbf{y}^TA\mathbf{y}\pm\sqrt{(\mathbf{y}^TA\mathbf{y})^2-4\mathbf{y}^T(D-I)\mathbf{y}}}{2}.\]
We work by cases:
\begin{enumerate}
    \item Assume that $(\mathbf{y}^TA\mathbf{y})^2<4\mathbf{y}^T(D-I)\mathbf{y}$.
    Then
    \begin{align*}
        |\mu|=\frac{(\mathbf{y}^TA\mathbf{y})^2}{4}+\mathbf{y}^T(D-I)\mathbf{y}-\frac{(\mathbf{y}^TA\mathbf{y})^2}{4}=\mathbf{y}^TD\mathbf{y}-1\geq 2-1=1.
    \end{align*}
    \item Assume that $(\mathbf{y}^TA\mathbf{y})^2=4\mathbf{y}^T(D-I)\mathbf{y}$.
    Then
     \[|\mu|=\Bigg|\frac{\mathbf{y}^TA\mathbf{y}}{2}\Bigg|=\Bigg|\sqrt{\mathbf{y}^T(D-I)\mathbf{y}}\Bigg|\geq \sqrt{2}\geq1.\]
     \item Assume that $(\mathbf{y}^TA\mathbf{y})^2>4\mathbf{y}^T(D-I)\mathbf{y}$.
     If $\mu=\frac{\mathbf{y}^TA\mathbf{y}+\sqrt{(\mathbf{y}^TA\mathbf{y})^2-4\mathbf{y}^T(D-I)\mathbf{y}}}{2}$, then
    \[|\mu|=\Bigg|\frac{\mathbf{y}^TA\mathbf{y}+\sqrt{(\mathbf{y}^TA\mathbf{y})^2-4\mathbf{y}^T(D-I)\mathbf{y}}}{2}\Bigg|\geq\Bigg|\frac{\mathbf{y}^TA\mathbf{y}}{2}\Bigg|=|\sqrt{\mathbf{y}^T(D-I)\mathbf{y}}|\geq 1.\]
    Assume $\mu=\frac{\mathbf{y}^TA\mathbf{y}-\sqrt{(\mathbf{y}^TA\mathbf{y})^2-4\mathbf{y}^T(D-I)\mathbf{y}}}{2}$.
    Define the function \[g(x)=\Bigg|\frac{x-\sqrt{x^2-4r}}{2}\Bigg|\] on the domain $(-\infty,-2\sqrt{r}]\cup[2\sqrt{r},\infty)$ where $r=\mathbf{y}^T(D-I)\mathbf{y}$.
    This function is clearly non-increasing.
    We also know that $(\mathbf{y}^TA\mathbf{y})^2>4\mathbf{y}^T(D-I)\mathbf{y}$.
    So either $-\mathbf{y}^TA\mathbf{y}>2\sqrt{r}$ or $\mathbf{y}^TA\mathbf{y}>2\sqrt{r}$.
    
    First assume that $-\mathbf{y}^TA\mathbf{y}>2\sqrt{r}$.
    Thus, $\mathbf{y}^TA\mathbf{y}<0$.
    Since $r\geq 1$, we know that $\mathbf{y}^TA\mathbf{y}<2\sqrt{r}$.
    Since $g$ is non-increasing, then $g(\mathbf{y}^TA\mathbf{y})>g(2\sqrt{r})\geq 1$.
    Thus, $|\mu|\geq 1$.
    
    Now assume that $\mathbf{y}^TA\mathbf{y}>2\sqrt{r}$.
    Note that if $g(\mathbf{y}^TA\mathbf{y})=1$, then
    \begin{align*}
        1&=\frac{\mathbf{y}^TA\mathbf{y}-\sqrt{(\mathbf{y}^TA\mathbf{y})^2-4\mathbf{y}^T(D-I)\mathbf{y}}}{2}\\
        \sqrt{(\mathbf{y}^TA\mathbf{y})^2-4\mathbf{y}^T(D-I)\mathbf{y}}&=\mathbf{y}^TA\mathbf{y}-2\\
        (\mathbf{y}^TA\mathbf{y})^2-4\mathbf{y}^T(D-I)\mathbf{y}&=(\mathbf{y}^TA\mathbf{y})^2-4\mathbf{y}^TA\mathbf{y}+4\\
        \mathbf{y}^T(A-(D-I))\mathbf{y}&=1\\
        \mathbf{y}^T(D-A)\mathbf{y}&=0.
    \end{align*}
    Thus, if $g(\mathbf{y}^TA\mathbf{y})=1$, then $\mathbf{y}^TA\mathbf{y}=\mathbf{y}^TD\mathbf{y}$.
    Further, we know that the Laplacian is positive semidefinite \cite{chen2016characterizing}, so $\mathbf{y}^TD\mathbf{y}\geq\mathbf{y}^TA\mathbf{y}$.
    Since $g$ is non-increasing, then $g(\mathbf{y}^TA\mathbf{y})\geq 1$ for all $\mathbf{y}$.
    So $|\mu|\geq 1$.
\end{enumerate}
Thus, $|\mu|\geq 1$ for all $\mu\in\sigma(B)$.
\end{proof}

\begin{remark}
An alternate proof of Theorem \ref{thm:smallest-eig} without using $K$ can be found in \cite{torres2020non}.
\end{remark}

With the fact that the modulus of all $\mu\in\sigma(B)$ are bounded below by one, we turn our attention to the spectral radius.
Our first goal is to show that the spectral radius of $B$ is strictly greater than 1 when $G$ is not a cycle and has no dangling vertexs.

\begin{proposition}
Let $G$ be a connected graph such that $G$ is not a tree or cycle and $d_{\min}\geq 2$. Then $\rho(K)>1$.
\label{thm:lower-bound-spectral-radius}
\end{proposition}

\begin{proof}
Let $B$ be the non-backtracking matrix.
Note that $\rho(B)=\rho(K)$ and $\rho(K)\geq 1$ since $\sigma(K)\subset\sigma(B)$ and $1\in\sigma(K)$.
Assume that $G$ is not $d$-regular.
Since $B$ is nonnegative and irreducible by Proposition \ref{thm:irreducible} and $d_{\min}\geq 2$, then \cite{liu2010some} tells us that $\rho(B)>\min_i\sum_jb_{ij}\geq 1$.
Assume that $G$ is $d$-regular.
Since $G$ is not a cycle, $d\geq 3$.
Then by Theorem \ref{kemp-reg-spectrum},
\[\rho(B)=\frac{d+\sqrt{d^2-4(d-1)}}{2}\geq\frac{3+1}{2}=2>1.\]
So $\rho(B)>1$.
Since $\rho(B)=\rho(K)$, then $\rho(K)>1$.
\end{proof}

We now turn our attention to upper bounds on the spectrum of $B$.
To do this, we employ Proposition \ref{thm:mu-equation}.
Recall that the necessary condition for the proposition was that $\mathbf{x}^T\mathbf{y}\neq 0$ for some eigenvalue-eigenvector pairs $(\lambda,\mathbf{x})$ and $(\mu,\begin{bmatrix}-\mu \mathbf{y}&\mathbf{y}\end{bmatrix}^T)$ of $A$ and $K$ respectively.
To obtain this condition, we show that $\mathbf{y}$ is a positive vector for the eigenvalue-eigenvector pair $(\rho(K),\begin{bmatrix}-\rho(K)\mathbf{y}&\mathbf{y}\end{bmatrix}^T)$.

\begin{lemma}
Let $G$ be a connected graph such that $G$ is not a tree or cycle and $d_{\min}\geq 2$. Then $y$ is positive for the eigenvector $\begin{pmatrix}-\rho(K)\mathbf{y}&\mathbf{y}\end{pmatrix}^T$.
\label{thm:y-positive}
\end{lemma}

\begin{proof}
Let $K\begin{bmatrix}-\rho(K)\mathbf{y}&\mathbf{y}\end{bmatrix}^T=\rho(K)\begin{bmatrix}-\rho(K)\mathbf{y}&\mathbf{y}\end{bmatrix}^T$.
Now scale $\mathbf{y}$ such that there exists some $y_k>0$.
By the Perron-Frobenius Theorem and Theorem \ref{thm:almost-similarity}, $T^T\mathbf{y}\succ \rho(K)S\mathbf{y}$ or $\rho(K)S\mathbf{y}\succ T^T\mathbf{y}$, where $S$ and $T$ are defined as in equation \ref{eqn:s-t}.
First assume that $T^T\mathbf{y}\succ \rho(K)S\mathbf{y}$.
From the definitions of $T$ and $S$, we get that $y_i\geq\rho(K)y_j$ for all $i\sim j$.
Choose $y_k>0$.
Thus for all $i$ such that $i\sim k$, $\frac{y_i}{y_k}\geq \rho(K)$.
So $y_i\neq0$ for all $i\sim k$ since $\rho(K)>1$ by Proposition \ref{thm:lower-bound-spectral-radius}.
Then by similar argument, $\frac{y_k}{y_i}\geq\rho(K)$.
This implies $\frac{y_k}{y_i}=1$ which is a contradiction since $\rho(K)>1$.

Assume that $\rho(S)\mathbf{y}\succ T^T\mathbf{y}$.
Again, choose $y_k>0$.
So for all $i$ such that $i\sim k$,
$\rho(K)\frac{y_i}{y_k}>1$.
Since $\rho(K)>1$, then $y_i$ must be positive.
Since $G$ is connected, by induction we get that $y_j>0$ for all vertices $j$.
Thus $y$ is positive.
\end{proof}

\begin{theorem}
Let $G$ be a connected graph with $A$ the adjacency matrix and $B$ the non-backtracking matrix. If $\rho(A)\geq2\sqrt{\mathbf{x}^T(D-I)\mathbf{y}}$, then
\[\rho(B)\leq\frac{\rho(A)+\sqrt{\rho(A)^2-4(d_{\min}-1)}}{2}.\]
\label{thm:upper-bound}
\end{theorem}

\begin{proof}
Let $G$ be a tree.
Then $\rho(B)=0$ by Theorem \ref{thm:b-tree} and $\rho(B)\leq\frac{\rho(A)+\sqrt{\rho(A)^2-4(d_{\min}-1)}}{2}$.

Now assume that $G$ is a cycle.
Then $\rho(B)=1$ by Theorem \ref{thm:cycle-spectrum} and $\frac{2+\sqrt{4-4}}{2}=1\geq \rho(B)$.
Now assume that $G$ is a cycle with dangling vertexs.
By Corollary \ref{thm:adding-trees} $\rho(B)$ is the spectral radius of $G\backslash S$ where $S$ is the set of dangling vertexs.
Then by the work above the result holds.

Assume that $G$ is not a tree or cycle and $d_{\min}\geq 2$.
Note that $\rho(B)\geq\rho(K)$.
By the Perron-Frobenius theorem and Lemma \ref{thm:y-positive}, $\mathbf{x}^T\mathbf{y}\neq 0$.
Thus by Proposition \ref{thm:mu-equation},
\[\rho(B)\leq\rho(K)=\frac{\rho(A)\pm\sqrt{\rho(A)^2-4\mathbf{x}^T(D-I)\mathbf{y}}}{2}\leq\frac{\rho(A)+\sqrt{\rho(A)^2-4(d_{\min}-1)}}{2}.\]
Now assume that $G$ is not a tree or cycles and has at least one dangling vertex.
By Corollary \ref{thm:adding-trees} $\rho(B)$ is the spectral radius of $G\backslash S$ where $S$ is the set of all dangling vertexs.
Then by the same argument as above the result holds.

\end{proof}

There are many bounds shown for the spectral radius of the adjacency matrix of a graph \cite{das2004some,liu2008upper,nikiforov2007bounds,stanic2015inequalities}.
Using a bound provided by Das and Kumar \cite{das2004some}, we get a simple bound on the spectral radius of $B$ dependent on the minimum degree $d_{\min}$, number of vertexs $n$, and number of edges $m$.

\begin{corollary}
Let $G$ be a connected graph with $A$ the adjacency matrix and $B$ the non-backtracking matrix.
If $\rho(A)\geq 2\sqrt{\mathbf{x}^T(D-I)\mathbf{y}}$, then
\[\rho(B)\leq\frac{\sqrt{2m-n-1}+\sqrt{2m-n-4d_{\min}+3}}{2}.\]
\end{corollary}

\subsection{Bipartite Graphs}

A bipartite graph is a graph where all vertices can be divided into two subsets, $V_1$ and $V_2$, where vertices in $V_1$ only connect to vertices in $V_2$ and vice versa.
It is widely known that the spectrum of the adjacency matrix can determine the bipartiteness of a graph.

\begin{theorem}[\cite{brouwer2011spectra}]
Let $G$ be a graph and $A$ its associated adjacency matrix.
\begin{enumerate}[(i)]
\item $G$ is bipartite if and only if, for each eigenvalue $\lambda\in\sigma(A)$, $-\lambda\in\sigma(A)$ with the same multiplicity.
\item If $G$ is connected and $\lambda_1$ is the largest eigenvalue of $A$, then $G$ is bipartite if and only if $-\lambda_1$ is an eigenvalue of $A$.
\end{enumerate}
\label{a-bipartite}
\end{theorem}

The spectrum of the non-backtracking matrix $B$ also can indicate whether a graph is bipartite in the same way.
Additionally, the same properties hold for the spectrum of $K$ when $G$ is bipartite.

\begin{theorem}
Let $G$ be a connected graph, $B$ its associated non-backtracking matrix, and $K$ defined as above. Then the following are equivalent:
\begin{enumerate}[(i)]
\item $G$ is a bipartite graph,
\item $\sigma(K)$ is symmetric,
\item $\sigma(B)$ is symmetric,
\item $\lambda_n=-\lambda_1$ where $\lambda_i\in\sigma(K)$, and
\item $\mu_n=-\mu_1$ where $\mu_i\in\sigma(B)$,
\item $-1\in\sigma(K)$.
\end{enumerate}
\label{b-bipartite}
\end{theorem}

\begin{proof}

\emph{$(i)\rightarrow(ii)$}: Assume that $G$ is bipartite. If $G$ is bipartite, then the adjacency matrix $A$ can be written as $\begin{bmatrix}0&A_2\\A_1&0\end{bmatrix}$ (see \cite{brouwer2011spectra}). We know that $A=TS$ by equation \ref{eqn:s-t-identities}. Thus for a bipartite graph, we define the following matrices:
\begin{align*}
T_1\colon=\begin{cases}1&i_1\mapsto(i_1,i_2)\\0&\text{otherwise}\end{cases},&&
T_2\colon=\begin{cases}1&i_2\mapsto(i_2,i_1)\\0&\text{otherwise}\end{cases},&&
S_1\colon=\begin{cases}1&(i_2,i_1)\mapsto i_1\\0&\text{otherwise}\end{cases},&&
S_2\colon=\begin{cases}1&(i_1,i_2)\mapsto i_2\\0&\text{otherwise}\end{cases}.
\end{align*}
In these matrices, $i_j$ represents a vertex in partition $j$ and $(i_j,i_k)$ represents an edge from partition $j$ to partition $k$.
Thus by simple computation, we see that $A=\begin{bmatrix}0&T_1S_2\\T_2S_1&0\end{bmatrix}$ and the matrix $C=\begin{bmatrix}0&S_2T_2\\S_1T_1&0\end{bmatrix}$ where $C$ is defined as in equation \ref{eqn:c-matrix}.
Hence, the edges are also divided into two edge partitions.

In order to compute $B$, we also define a matrix
\begin{align*}
\tau_1\colon=\begin{cases}1&(w,x)_1\mapsto(y,z)_2\\0&\text{otherwise}\end{cases}\text{ and }&
\tau_2\colon=\begin{cases}1&(w,x)_2\mapsto(y,z)_1\\0&\text{otherwise}\end{cases},
\end{align*}
where $(w,x)_j$ represents an edge in edge partition $j$. We then see that $\tau=\begin{bmatrix}0&\tau_2\\\tau_1&0\end{bmatrix}$.
So the matrix $B=\begin{bmatrix}0&S_2T_2-\tau_2\\S_1T_1-\tau_1&0\end{bmatrix}$.
Defining $B_j=S_jT_j-\tau_j$, we get that $B=\begin{bmatrix}0&B_2\\B_1&0\end{bmatrix}$.

Let $\begin{bmatrix}\mathbf{x}&\mathbf{y}\end{bmatrix}^T$ be an eigenvector of $B$ with corresponding eigenvalue $\mu$. Then
\begin{align*}
B\begin{bmatrix}\mathbf{x}&\mathbf{y}\end{bmatrix}^T&=\mu\begin{bmatrix}\mathbf{x}&\mathbf{y}\end{bmatrix}^T\\
\begin{bmatrix}B_2\mathbf{y}&B_1\mathbf{x}\end{bmatrix}^T&=\mu\begin{bmatrix}\mathbf{x}&\mathbf{y}\end{bmatrix}^T.
\end{align*}
Consider the vector $\begin{bmatrix}\mathbf{x}&-\mathbf{y}\end{bmatrix}^T$. We see that
\begin{align*}
B\begin{bmatrix}\mathbf{x}&-\mathbf{y}\end{bmatrix}^T&=\begin{bmatrix}-B_2\mathbf{y}&B_1\mathbf{x}\end{bmatrix}^T
=\begin{bmatrix}-\mu \mathbf{x}&\mu \mathbf{y}\end{bmatrix}^T
=-\mu\begin{bmatrix}\mathbf{x}&-\mathbf{y}\end{bmatrix}^T.
\end{align*}
So $-\mu$ is an eigenvalue of $B$ with eigenvector $\begin{bmatrix}\mathbf{x}&-\mathbf{y}\end{bmatrix}^T$. Hence the spectrum of $B$ is symmetric around 0.

\emph{$(ii)\rightarrow(iii)$}:  Recall that $\sigma(B)=\sigma(K)\cup\{\pm 1\}$. Since the set of $\pm 1$ is symmetric by Ihara's theorem, $\sigma(K)$ is symmetric.


\emph{$(iii)\rightarrow(iv)$}: Assume $\sigma(K)$ is symmetric. Then $\lambda_n=-\lambda_1$ where $\lambda_1$ is the spectral radius of $\sigma(K)$.

\emph{$(iv)\rightarrow(v)$}: If $G$ is a tree, then the result follows trivially from Theorem \ref{thm:b-tree}.

If $G$ is not a tree then $\sigma(K)\subseteq\sigma(B)$. Since $1\in\sigma(K)$, we know that $\rho(K)\geq 1$. Thus, $\rho(B)\geq 1$. If $\rho(B)=1$, then the dominant eigenvalue of $B$ must be $\mu_1=\pm 1$. From Ihara's Theorem, $\pm 1\in\sigma(B)$. 

If $\rho(B)\neq 1$, then $\mu_1\in\sigma(K)$ by Ihara's Theorem and $\mu_1=\lambda_1$. Thus, $\lambda_n=-\mu_1$. Then, we know that $\mu_n=-\mu_1$.

\emph{$(v)\rightarrow(i)$}: If $G$ is tree or cycle, the result holds from Theorems \ref{thm:b-tree} and \ref{thm:cycle-spectrum} respectively.
Assume that $G$ is not a tree or cycle and $\mu_n=-\mu_1$.
Then $|\mu_n|=|\mu_1|$.
By Proposition \ref{thm:irreducible} tells us $B$ is irreducible. Then from the Perron-Frobenius theorem, we know that the period $d$ of $B$ must be at least 2. Further, since $\mu_n=-\mu_1$, $d=2k$ for some $k\in\{1,2,3,...\}$. Then every cycle must be of even length, so $B$ is the adjacency matrix of a directed bipartite graph. So $B=\begin{bmatrix}0&B_2\\B_1&0\end{bmatrix}$, implying that $A=\begin{bmatrix}0&A_2\\A_1&0\end{bmatrix}$. Hence $G$ is bipartite.

\emph{$(iii)\rightarrow(vi)$:} Recall that $1\in\sigma(K)$ from Theorem \ref{thm:props-of-k}(ii). Since $\sigma(K)$ is symmetric, $-1\in\sigma(K)$.

\emph{$(vi)\rightarrow(i)$:} Since $-1\in\sigma(K)$, we know that $\begin{pmatrix}\mathbf{y}&\mathbf{y}\end{pmatrix}$ is the associated eigenvector by Theorem \ref{thm:props-of-k}(i).
Thus, $A\mathbf{y}+(D-I)\mathbf{y}=-\mathbf{y}$ implying that $D^{-1}A\mathbf{y}=-\mathbf{y}$.
So $-1\in\sigma(D^{-1}A)$.
Note that $D^{-1}A$ is the transition probability matrix of $G$ with spectral radius 1.
Thus, $G$ must be bipartite.
\end{proof}

\begin{remark}
An alternate proof of Theorem \ref{b-bipartite}(vi) can be derived directly as a corollary of Proposition 4.15 of \cite{torres2020non}.
\end{remark}

Recall in Theorem \ref{thm:b-tree} we calculated $\sigma(B)$ when $G$ is a tree.
Since $\sigma(K)\supset\sigma(B)$ in the case of a tree, we can use the properties of bipartite graphs to classify $\sigma(K)$ for trees as well.

\begin{corollary}
Let $G$ be a tree. Let $K$ be defined as above. Then $1,-1,0\in\sigma(K)$. The eigenvalue 0 has multiplicity at least $r$, where $r$ is the number of leaves in $G$.
\label{k-tree}
\end{corollary}



\bibliographystyle{plain}
\bibliography{ref}

\end{document}